\documentclass[11pt]{article}
\usepackage{amsmath,amsthm,amsfonts,amssymb,enumerate}
\usepackage[colorlinks=true,citecolor=black,linkcolor=black,urlcolor=blue]{hyperref}
\usepackage[lmargin=30mm,rmargin=30mm,bmargin=30mm,tmargin=30mm]{geometry}

\theoremstyle{plain}
\newtheorem{theorem}{Theorem}
\newtheorem{lemma}[theorem]{Lemma}

\newtheorem{conjecture}[theorem]{Conjecture}

\newenvironment{definition}[1][Definition]{\begin{trivlist}\item[\hskip \labelsep {\bfseries #1}]}{\end{trivlist}}
\theoremstyle{remark}
\newtheorem{claim}{Claim}

\DeclareMathOperator{\tw}{tw}

\DeclareMathOperator{\Kneser}{Kneser}

\begin{document}

\title{\bf{Treewidth of the Kneser Graph and the Erd\H{o}s-Ko-Rado Theorem}}
\author{Daniel J. Harvey\footnote{Department of Mathematics and Statistics, The University of Melbourne, Melbourne, Australia  (\texttt{d.harvey@pgrad.unimelb.edu.au}). Supported by an Australian Postgraduate Award.} \,and David R. Wood\footnote{School of Mathematical Sciences, Monash University, Melbourne, Australia (\texttt{david.wood@monash.edu}). Supported by the Australian Research Council.}}
\maketitle

\begin{abstract}
Treewidth is an important and well-known graph parameter that measures the complexity of a graph. The \emph{Kneser graph} $\Kneser(n,k)$ is the graph with vertex set $\binom{[n]}{k}$, such that two vertices are adjacent if they are disjoint. We determine, for large values of $n$ with respect to $k$, the exact treewidth of the Kneser graph. In the process of doing so, we also prove a strengthening of the Erd\H{o}s-Ko-Rado Theorem (for large $n$ with respect to $k$) when a number of disjoint pairs of $k$-sets are allowed. 
\end{abstract}

\section{Introduction}
\label{section:intro} 

A \emph{tree decomposition} of a graph $G$ is a pair $(T, (B_{x} \subset V(G) :x \in V(T)))$ where $T$ is a tree and $(B_{x} \subseteq V(G): x \in V(T))$ is a collection of sets, called \emph{bags}, indexed by the nodes of $T$. The following properties must also hold:
\begin{itemize}
\item for each $v \in V(G)$, the nodes of $T$ that index the bags containing $v$ induce a non-empty connected subtree of $T$,
\item for each $vw \in E(G)$, there exists some bag containing both $v$ and $w$.
\end{itemize}

The \emph{width} of a tree decomposition is the size of the largest bag, minus 1. The \emph{treewidth} of a graph $G$, denoted $\tw(G)$, is the minimum width of a tree decomposition of $G$.

Treewidth is an important concept in modern graph theory. Treewidth was initially defined by Halin \cite{Halin76} (with different nomenclature to the modern standard) and then later by Robertson and Seymour \cite{GMII}, who used it in their famous series of papers proving the Graph Minor Theorem \cite{GMall}. The treewidth of a graph essentially describes how ``tree-like" it is, where lower treewidth implies a more ``tree-like" structure. (A forest has treewidth at most 1, for example.) Treewidth is also of key interest in the field of algorithm design---for example, treewidth is a key parameter in fixed-parameter tractability \cite{Bodlaender-AC93}. \newline

Let $[n] = \{1, \dots, n\}$. For any set $S \subseteq [n]$, a subset of $S$ of size $k$ is called a \emph{$k$-set}, or occasionally a $k$-set in $S$. Let $\binom{S}{k}$ denote the set of all $k$-sets in $S$. We say two sets \emph{intersect} when they have non-empty intersection.

The \emph{Kneser graph} $\Kneser(n,k)$ is the graph with vertex set $\binom{[n]}{k}$, such that two vertices are adjacent if they are disjoint.

Kneser graphs were first investigated by Kneser \cite{kneser}. The chromatic number of $\Kneser(n,k)$ was shown to be $n-2k+2$ by Lov\'{a}sz \cite{lazlo}, as Kneser originally conjectured. This was an important proof due to the development of the topological methods involved. Many other proofs of this result have been found, for example consider \cite{Ziegler}, which gives a more combinatorial version. The Kneser graph is also of interest with regards to fractional chromatic number \cite{fgt}. The famous Erd\H{o}s-Ko-Rado Theorem \cite{EKR} has a well-known relationship to the Kneser graph, as does the generalisation to cross-intersecting families by Pyber \cite{Pyber}. We discuss these in more detail in Section~\ref{section:prelim}, and shall use both of these results to prove the following two theorems about the treewidth of the Kneser graph.

\begin{theorem}
\label{theorem:main}
Let $G$ be a Kneser graph with $n \geq 4k^2 - 4k + 3$ and $k \geq 3$. Then $$\tw(G) = \binom{n-1}{k}-1.$$
\end{theorem}
This theorem is our main result, giving an exact answer for the treewidth of the Kneser graph when $n$ is sufficiently large. In order to prove this, we show that $\binom{n-1}{k}-1$ is both an upper bound and lower bound on the treewidth. We construct a tree decomposition directly in Section~\ref{section:upper} to prove an upper bound. In Section~\ref{section:lbone} we prove the lower bound by using the relationship between treewidth and separators.


We also prove the following more precise result when $k=2$.

\begin{theorem}
\label{theorem:ktwo}
Let $G$ be a Kneser graph with $k=2$. Then \newline
$$\tw(G) =
\begin{cases}
0 &\text{ if } n \leq 3 \\
1 &\text{ if } n = 4 \\
4 &\text{ if } n = 5 \\
\binom{n-1}{2} - 1 &\text{ if } n \geq 6.
\end{cases}$$
\end{theorem}

The upper bounds for Theorem~\ref{theorem:ktwo} are proved in Section~\ref{section:upper}, and the lower bounds in Section~\ref{section:ktwolow}.


Finally, in the process of proving Theorem~\ref{theorem:main}, we prove the following generalisation of the Erd\H{o}s-Ko-Rado Theorem (Theorem~\ref{theorem:ekr} in Section~\ref{section:prelim}), which says that if $n \geq 2k$ and $H$ is a complete subgraph in the complement of $\Kneser(n,k)$ then $|H| \leq \binom{n-1}{k-1}.$ We prove the same bound for balanced complete multipartite graphs. 

\begin{theorem}
\label{theorem:turan}
Say $p \in [\frac{2}{3},1)$ and $n \geq \max(4k^2 - 4k +3, \frac{1}{1-p}(k^2 -1)+2)$. If $H$ is a complete multipartite subgraph of the complement of $\Kneser(n,k)$ such that no colour class contains more than $p|H|$ vertices, then $|H| \leq \binom{n-1}{k-1}$.
 
\end{theorem}

Note that similar, but incomparable, generalisations of the Erd\H{o}s-Ko-Rado Theorem have recently been explored in \cite{almostisect,sperner,scott}.
Theorem~\ref{theorem:turan} is proven in Section~\ref{section:lbone}, since it follows almost directly from our proof of the lower bound on the treewidth of a Kneser graph.

\section{Basic Definitions and Preliminaries}
\label{section:prelim}

From now on, we refer to the graph $\Kneser(n,k)$ as $G$, with $n$ and $k$ implicit. 


Let $\Delta(H)$ be the maximum degree of a graph $H$ and $\delta(H)$ be the minimum degree of a graph $H$. Also let $\alpha(H)$ be the size of the largest independent set of $H$, where an \emph{independent set} is a set of pairwise non-adjacent vertices. If $k=1$, then $G$ is the complete graph. If $n < 2k$ then $G$ has no edges. If $n=2k$ then $G$ is an induced matching. From now on, we shall assume that $n \geq 2k+1$ and $k \geq 2$, since the treewidth is trivial in the other cases.

In order to prove a lower bound on the treewidth of the Kneser graph, we use a known result about the relationship between treewidth and separators.

\begin{definition}
Given a constant $p \in [\frac{2}{3},1)$, a \emph{$p$-separator} (of order $k$) is a set $X \subset V(G)$ such that $|X| \leq k$ and no component of $G-X$ contains more than $p|G-X|$ vertices. 
\end{definition}

\begin{theorem} \cite{GMII}
\label{theorem:sep}
For each $p \in [\frac{2}{3},1)$, every graph $G$ has a $p$-separator of order $\tw(G)+1$.
\end{theorem}

It can easily be shown that we can partition the components of $G-X$ into two parts, such that the components in a part contain, in total, at most $p|G-X|$ vertices. 
This gives the following lemma.



\begin{lemma}
\label{lemma:part}
Let $X$ be a $p$-separator. Then $V(G-X)$ can be partitioned into two parts $A$ and $B$, with no edge between $A$ and $B$, such that
\begin{itemize}
\item $(1-p)|G-X| \leq |A| \leq \frac{1}{2}|G-X|$,
\item $\frac{1}{2}|G-X| \leq |B| \leq p|G-X|$.
\end{itemize}
\end{lemma}


We use a few important well known combinatorial results.





\begin{theorem}[Erd\H{o}s-Ko-Rado \cite{EKR,EKRSimple}]
\label{theorem:ekr}
Let $G$ be $\Kneser(n,k)$ for some $n \geq 2k$. Then $\alpha(G) = \binom{n-1}{k-1}.$
\end{theorem}

The original Erd\H{o}s-Ko-Rado Theorem defines $\mathcal{A}$ as a set of $k$-sets in $[n]$, such that the $k$-sets of $\mathcal{A}$ pairwise intersect. Our formulation in terms of vertices in the Kneser graph is clearly equivalent. 
We will use Theorem~\ref{theorem:ekr} when determining an upper bound for $\tw(G)$.

The second major result is by Pyber \cite{Pyber}. 
Let $\mathcal{A}$ and $\mathcal{B}$ be sets of vertices of the Kneser graph $G$, such that for all $v \in \mathcal{A}$ and $w \in \mathcal{B}$ the pair $vw$ is not an edge. Then we say the pair $(\mathcal{A}, \mathcal{B})$ are \emph{cross-intersecting families}.

\begin{theorem}[Erd\H{o}s-Ko-Rado for Cross-Intersecting Families \cite{Pyber, MatsuToku}]
\label{theorem:ekrx}
Let $n \geq 2k$ and let $(\mathcal{A},\mathcal{B})$ be cross-intersecting families in the Kneser graph $G$. Then $|\mathcal{A}||\mathcal{B}| \leq \binom{n-1}{k-1}^2$. If $n \geq 2k+1$ and $(\mathcal{A},\mathcal{B})$ are cross-intersecting families such that $|\mathcal{A}||\mathcal{B}| = \binom{n-1}{k-1}^2$, then $\mathcal{A} = \mathcal{B} = \{v | i \in v\}$ for a fixed element $i \in [n]$.
\end{theorem}

As with Theorem~\ref{theorem:ekr}, the original formulation by Pyber of Theorem~\ref{theorem:ekrx} is more general. We have given the result in an equivalent form that is sufficient for our requirements. 



Let $X$ be a $\frac{2}{3}$-separator and $A$,$B$ the parts of the vertex partition of $G-X$ as in Lemma~\ref{lemma:part}. Now for all $v \in A$ and $w \in B$, $v$ and $w$ are in different components and as such are non-adjacent. So $(A,B)$ are cross-intersecting families. We know $|A| = c|G-X|$ where $\frac{1}{3} \leq c \leq \frac{1}{2}$. By Theorem~\ref{theorem:ekrx}, it follows that $c(1-c)|G-X|^{2} \leq \binom{n-1}{k-1}^2$. It follows that $|G-X| \leq \frac{3}{\sqrt{2}}\binom{n-1}{k-1}$. (We leave the precise calculation to the reader.) This gives a lower bound on $|X|$, and as such a lower bound on the treewidth (by Theorem~\ref{theorem:sep}). Hence $\tw(G) \geq \binom{n}{k} - \frac{3}{\sqrt{2}}\binom{n-1}{k-1} -1$.

However, note that the parts $A$ and $B$ of $V(G-X)$ are vertex disjoint, but that the definition of a pair of cross-intersecting families does not require this. In fact, Theorem~\ref{theorem:ekrx} shows that in the case where $|\mathcal{A}||\mathcal{B}|$ is maximised, $\mathcal{A}=\mathcal{B}$. We show we can do better than the above na\"{i}ve lower bound on $\tw(G)$ when $\mathcal{A}$ and $\mathcal{B}$ are disjoint. 



Before considering our final preliminary, we provide the following definitions. 
Consider all of the $a$-sets in $[b]$. Define the \emph{colexicographic} or \emph{colex} ordering on the $a$-sets as follows: if $x$ and $y$ are distinct $a$-sets, then $x < y$ when $\max(x-y) < \max(y-x)$. This is a strict total order. 
%
A set $X$ of $a$-sets in $[b]$ is \emph{first} if $X$ consists of the first $|X|$ $a$-sets in the colex ordering of all the $a$-sets in $[b]$.

Now consider the colex ordering of $a$-sets in $[b]$. All of the $a$-sets in $[i]$ (where $i<b$) come before any $a$-set containing an element greater than or equal to $i+1$. To see this, note if $x$ is an $a$-set in $[i]$ and $y$ is an $a$-set with $j \in y$ such that $j \geq i+1$, then $\max(x-y) \leq \max(x) \leq i$, and $\max(y-x) \geq j \geq i+1$ as $j \in y-x$. 
We will use this when determining the make-up of first sets in Section~\ref{section:lbone}.  

Let $X$ be a set of $a$-sets in $[b]$. For $c \leq a$, the \emph{$c$-shadow of $X$} is the set $\{ x : |x| = c,$ and $\exists y \in X$ such that $x \subseteq y\}$. That is, the $c$-shadow contains all $c$-sets that are contained within $a$-sets of $X$.
%
If $x$ is an $a$-set in $[b]$, let the \emph{complement} of $x$ be the $(b-a)$-set $y = [b]-x$. If $X$ is a set of $a$-sets on $[b]$, then the \emph{complement} of $X$ is $\overline{X} := \{ y: y$ is the complement of some $x \in X \}$. Note $|X| = |\overline{X}|$.

\begin{lemma}[A first set minimises the shadow \cite{Kruskal,Katona} (see \cite{Frankl2} for a short proof)]
\label{lemma:minshadow}
Let $X$ be a set of $a$-sets on $[b]$, $c \leq a$ and $S$ be the $c$-shadow of $X$. Suppose $|X|$ is fixed but $X$ is not. Then $|S|$ is minimised when $X$ is first.
\end{lemma}

This idea is also used by Pyber \cite{Pyber} and Matsumoto and Tokushige \cite{MatsuToku}. Intuitively, the shadow $S$ should be minimised whenever the $a$-sets of $X$ ``overlap" as much as possible, so that each $c$-set in $S$ is a subset of as many $a$-sets as possible.

\section{Upper Bound for Treewidth}
\label{section:upper}
This section proves the upper bounds on $\tw(G)$ in Theorems~\ref{theorem:main} and \ref{theorem:ktwo}.

In both Theorem~\ref{theorem:main} and~\ref{theorem:ktwo}, the upper bound is almost always $\binom{n-1}{k}-1$. The only exceptions are the trivial cases (when $n \leq 2k$), and the case when $k=2$ and $n=5$, which is the Petersen graph. The Petersen graph is well-known to have treewidth $4$ (\cite{peterstree}, for example). What follows is a general upper bound on the treewidth of any graph, which is sufficient to prove the remaining cases.

\begin{lemma}
\label{lemma:gub}
If $H$ is any graph, then $\tw(H) \leq \max\{\Delta(H), |V(H)| - \alpha(H) - 1\}$.
\end{lemma}
\begin{proof}
Let $\alpha := \alpha(H)$. We shall construct a tree decomposition with underlying tree $T$, where $T$ is a star with $\alpha(H)$ leaves. Let $R$ be the bag indexed by the central node of $T$, and label the other bags $B_1, \dots, B_{\alpha}$. Let $X := \{x_1, \dots x_\alpha \}$ be a maximum independent set in $H$. Let $R:= V(H)-X$ and $B_i := N(x_{i}) \cup \{x_i\}$ for all $i \in \{1, \dots, \alpha\}$. We now show this is a tree decomposition:

Any vertex not in $X$ is contained in $R$. Given the structure of the star, any induced subgraph containing the central node is connected. Alternatively, if a vertex is in $X$, then it appears only in bags indexed by leaves. However, since $X$ is an independent set, $x_i \in X$ appears only in $B_i$, not in any other bag $B_j$. A single node is obviously connected. 
If $vw$ is an edge of $H$, then at most one of $v$ and $w$ is in $X$. Say $v = x_i \in X$. Then $v,w$ both appear in the bag $B_i$. Otherwise neither vertex is in $X$, and both vertices appear in $R$.

So this is a tree decomposition. The size of $R$ is $|V(H)| - \alpha(H)$. The size of $B_{i}$ is the degree of $x_i$, plus one, which is at most $\Delta(H) + 1$. From here our lemma is proven.
\end{proof}

We now consider this result for the Kneser graph itself.
\begin{lemma}
If $G$ is a Kneser graph with $k \geq 2$ and $n \geq 2k+1$, then $\tw(G) \leq \binom{n}{k-1} - 1$.
\end{lemma}
\begin{proof}
By Lemma~\ref{lemma:gub} and Theorem~\ref{theorem:ekr}, and since $n \geq 2k+1$, $$\tw(G) \leq \max\left\{\Delta(G), |V(G)| - \alpha(G) - 1\right\} = \max\left\{\binom{n-k}{k}, \binom{n}{k} - \binom{n-1}{k-1} - 1\right\}.$$ Since $k \geq 2$,  $\tw(G) \leq \binom{n-1}{k}-1$, as required.
\end{proof}




\section{Separators in the Kneser Graph}
\label{section:lbone}
To complete the proof of Theorem~\ref{theorem:main}, it is sufficient to prove a lower bound on the treewidth. The following lemma, together with Theorem~\ref{theorem:sep}, provides this. It is the heart of the proof of Theorem~\ref{theorem:turan}.

\begin{lemma}
\label{lemma:keylemma}
Let $X$ be a $p$-separator of the Kneser graph $G$. If $n \geq \max(4k^2 - 4k +3, \frac{1}{1-p}(k^2 -1)+2)$, then $|X| \geq \binom{n-1}{k}$.
\end{lemma}


\begin{proof}
Assume, for the sake of a contradiction, that $|X| < \binom{n-1}{k}$. Then $|G-X| > \binom{n-1}{k-1}$. By Lemma~\ref{lemma:part}, $G-X$ has two parts $A$ and $B$ such that $(1-p)|G-X| \leq |A| \leq \frac{1}{2}|G-X|$ and $\frac{1}{2}|G-X| \leq |B| \leq p|G-X|$ and no edge has an endpoint in both $A$ and $B$. 

For a given element $i \in [n]$, let $A_i := \{v \in A: i \in v\}$. Also define $A_{-i} := \{v \in A: i \notin v\}$. So $A_i$ and $A_{-i}$ partition the set $A$, for any choice of $i$. Define analogous sets for $B$.

\begin{claim}
\label{claim:akth}
There exists some $i$ such that $|B_{i}| \geq \frac{1}{k}|B|$.
\end{claim}
\begin{proof}
As $|A| \geq (1-p)|G-X| > 0$, there is a vertex $v \in A$. Without loss of generality, $v = \{1, \dots, k\}$. Each $w \in B$ is not adjacent to $v$, and so $w$ and $v$ intersect. Thus each $w$ must contain at least one of $1, \dots, k$. Hence at least one of these elements appears in at least $\frac{1}{k}|B|$ of the vertices of $B$, as required.
\end{proof}

Without loss of generality, $|B_n| \geq \frac{1}{k}|B|$.

\begin{claim}
\label{claim:lbbn}
$|B_n| > \binom{n-3}{k-2} + \binom{n-2}{k-2}$.
\end{claim}
\begin{proof}
$|B| \geq \frac{1}{2}|G-X| \geq \frac{1}{2}\binom{n-1}{k-1}$. Then by Claim~\ref{claim:akth} and our subsequent assumption, $|B_{n}| \geq \frac{1}{k}|B| \geq \frac{1}{2k}|G-X| \geq \frac{1}{2k}\binom{n-1}{k-1}$. Assume for the sake of a contradiction that $|B_n| \leq \binom{n-3}{k-2} + \binom{n-2}{k-2}$.
So $$\frac{1}{2k}\binom{n-1}{k-1} \leq \binom{n-3}{k-2} + \binom{n-2}{k-2}.$$
Thus $$(n-1)! \leq 2k(k-1)(n-k)((n-3)! + (n-2)!).$$
Hence $$n^2 - 3n + 2 = (n-1)(n-2) \leq 2k(k-1)(2n-k-2) = 4k^{2}n -4kn -2k^3 - 2k^2 + 4k.$$
So $n^2 + (4k-4k^2 - 3)n + 2k^3 + 2k^2 - 4k + 2 \leq 0.$
Since $n \geq 4k^2 - 4k + 3$, it follows $2k^3 + 2k^2 - 4k + 2 \leq 0$. Given that $k \geq 1$, this provides our desired contradiction.
\end{proof}


Consider the set $\overline{A_{-n}}$, that is, the complements of the vertices in $A$ that do not contain $n$. So every set in $\overline{A_{-n}}$ contains $n$. Let $\overline{A_{-n}}^{*} := \{ \overline{v} - n: \overline{v} \in \overline{A_{-n}}\}$. That is, remove $n$ from each set in $\overline{A_{-n}}$. There is clearly a one-to-one correspondence between $(n-k)$-sets in $\overline{A_{-n}}$ and $(n-k-1)$-sets in $\overline{A_{-n}}^{*}$.

Similarly, define $B_n^* := \{ v-n : v \in B_n\}$. That is, remove from each vertex of $B_n$ the element $n$, which they all contain. The resultant sets are $(k-1)$-sets in $[n-1]$.

\begin{claim}
\label{claim:avshadow}
If $v^* \in B_n^*$ and $\overline{w}^* \in \overline{A_{-n}}^*$, then $v^* \not\subseteq \overline{w}^*$.
\end{claim}
\begin{proof}
Assume, for the sake of a contradiction, that $v^* \subseteq \overline{w}^*$. Then it follows that $v \subset \overline{w}$, by re-adding $n$ to both sets. Thus $v$ and $w$ are adjacent. However, $v \in B_n \subset B$ and $w \in A_n \subset A$, which is a contradiction.
\end{proof}

Let $S$ be the $(k-1)$-shadow of $\overline{A_{-n}}^*$. Hence if $v \in B_n^*$, then $v \notin S$, by Claim~\ref{claim:avshadow}.
So, it follows that $$B_n^* \subseteq \binom{[n-1]}{k-1} - S.$$

Hence we have an upper bound for $|B_n^*|$ when we take $|S|$ to be minimised. By Lemma~\ref{lemma:minshadow}, $|S|$ is minimised when $\overline{A_{-n}}^*$ is first. 

\begin{claim}
\label{claim:ubann}
$|A_{-n}| \leq \binom{n-3}{k-2}$.
\end{claim}
\begin{proof}
$|A_{-n}| = |\overline{A_{-n}}| = |\overline{A_{-n}}^*|$, so it is sufficient to show that $|\overline{A_{-n}}^*| \leq \binom{n-3}{k-2}$. Assume for the sake of contradiction that $|\overline{A_{-n}}^*| \geq \binom{n-3}{k-2} = \binom{n-3}{n-k-1}$. 

Firstly, we show that $|S| \geq \binom{n-3}{k-1}$. It is sufficient to prove this lower bound when $|S|$ is minimised. Hence we can assume that $\overline{A_{-n}}^*$ is first, and contains the first $\binom{n-3}{n-k-1}$ $(n-k-1)$-sets in the colexicographic ordering. That is, it contains all $(n-k-1)$-sets on $[n-3]$. This is because there are $\binom{n-3}{n-k-1}$ such sets, and they come before all other sets in the ordering. In that case, $S$ contains all $(k-1)$-sets in $[n-3]$. As all of the $(k-1)$-sets in $[n-3]$ are in $S$, it follows that $|S| \geq \binom{n-3}{k-1}$, as required.


Then it follows that $|B_n^*| \leq \binom{n-1}{k-1} - \binom{n-3}{k-1}  = \binom{n-3}{k-2} + \binom{n-2}{k-2}$. However,
$|B_n^*| = |B_n| > \binom{n-3}{k-2} + \binom{n-2}{k-2}$ by Claim~\ref{claim:lbbn}. This provides our desired contradiction.
\end{proof}

\begin{claim}
\label{claim:anb}
$|A_n| \geq \frac{k}{k+1}|A|$.
\end{claim}
\begin{proof}
First, show that $|A_n| \geq k|A_{-n}|$. Suppose otherwise, for the sake of a contradiction. By Claim~\ref{claim:ubann}, $|A| = |A_n| + |A_{-n}| < (k+1)|A_{-n}| \leq (k+1)\binom{n-3}{k-2}$. But $|A| \geq (1-p)|G-X|$. Hence $(1-p)\binom{n-1}{k-1} < (k+1)\binom{n-3}{k-2}$. Thus $(n-1)(n-2) < \frac{1}{1-p}(k+1)(k-1)(n-k) \leq \frac{1}{1-p}(k+1)(k-1)(n-2)$. Thus $n < \frac{1}{1-p}(k^2 - 1)+1$, which contradicts our lower bound on $n$.

Then $|A_n| \geq k|A_{-n}| = k(|A| - |A_n|)$. So $(k+1)|A_n| \geq k|A|$ as required.
\end{proof}

\begin{claim}
\label{claim:bnisb}
$B_n = B$.
\end{claim}
\begin{proof}
Suppose, for the sake of a contradiction, that there exists some vertex $v \in B$ such that $n \notin v$. So each $w \in A_n$ contains $n$ (by definition) and some element of $v$ (which is not $n$), since $vw$ is not an edge. Any vertex of $A_n$ can be constructed as follows---take element $n$, choose one of the $k$ elements of $v$, and choose the remaining $k-2$ elements from the remaining $n-2$ elements of $[n]$. Thus $$|A_n| \leq 1\cdot k\binom{n-2}{k-2}.$$ Note this is actually a weak upper bound, since we have counted some of the vertices of $A_n$ more than once. Recall $|A| \geq (1-p)|G-X| \geq (1-p)\binom{n-1}{k-1}$. So by Claim~\ref{claim:anb}, $$\frac{(1-p)k}{(k+1)}\binom{n-1}{k-1} \leq \frac{k}{k+1}|A| \leq k\binom{n-2}{k-2}.$$
Thus $\frac{n-1}{k-1} \leq \frac{1}{1-p}(k+1)$ and $n \leq \frac{1}{1-p}(k^2-1)+1$, which contradicts our lower bound on $n$. 
\end{proof}

\begin{claim}
\label{claim:anisa}
$A_n = A$.
\end{claim}
\begin{proof}
This follows by essentially the same argument as Claim~\ref{claim:bnisb}. Assume our claim does not hold and there exists $v \in A$ such that $n \notin v$. By Claim~\ref{claim:bnisb}, $|B_n| = |B| \geq \frac{1}{2}\binom{n-1}{k-1}$. There is an upper bound on $|B_n|$ equal to the upper bound on $|A_n|$ in the previous proof. Then
$$\frac{1}{2}\binom{n-1}{k-1} \leq |B| = |B_n| \leq k\binom{n-2}{k-2},$$ and so $n \leq 2k(k-1)+1.$ This contradicts our lower bound on $n$.
\end{proof}

Claims~\ref{claim:bnisb} and \ref{claim:anisa} show that every vertex in $G-X = A \cup B$ contains $n$. Thus $|G-X| \leq \binom{n-1}{k-1}$ and $|X| \geq \binom{n-1}{k}$, our desired contradiction. 
\end{proof}


By Lemma~\ref{lemma:keylemma}, if $X$ is a $\frac{2}{3}$-separator of the Kneser graph $G$ and $n \geq 4k^2 -4k + 3$, then $|X| \geq \binom{n-1}{k}$. Hence by Theorem~\ref{theorem:sep}, $\tw(G) \geq \binom{n-1}{k}-1$. This proves Theorem~\ref{theorem:main}. 

Also, Lemma~\ref{lemma:keylemma} allows us to prove Theorem~\ref{theorem:turan}.
\begin{proof}[Proof of Theorem~\ref{theorem:turan}]
Let $C_1, \dots, C_r$ be the colour classes of $H$ and recall $G = \Kneser(n,k)$. Let $X:= V(\overline{G}) - V(H)$, so that $X,C_1, \dots, C_r$ is a partition of the vertex set of $\overline{G}$ (and also $G$). In $G$ there are no edges between any pair $C_i,C_j$, and $|C_i| \leq p|H| = p|G-X|$ for each $i$. So $X$ is a $p$-separator of $G$, and $|X| \geq \binom{n-1}{k}$ by Lemma~\ref{lemma:keylemma}. Hence $|H| \leq \binom{n-1}{k-1}$.
\end{proof} 

\section{Lower Bound for Treewidth in Theorem~\ref{theorem:ktwo}} 
\label{section:ktwolow}
To complete our proof of Theorem~\ref{theorem:ktwo}, we need to obtain a lower bound on the treewidth when $k=2$. If $n \leq 4$, then Theorem~\ref{theorem:ktwo} is trivial. When $n=5$, then $G$ is the Petersen graph, which has a $K_5$-minor forcing $\tw(G) \geq 4$. Hence we may assume that $n \geq 6$. 

Assume, for the sake of a contradiction that $\tw(G) < \binom{n-1}{2} -1$. Let $(T, (B_x:x \in V(T)))$ be a minimum width tree decomposition for $G$, and normalise the tree decomposition such that if $xy \in E(T)$, then $B_x \not\subseteq B_y$ and $B_y \not\subseteq B_x$. By Theorem~\ref{theorem:sep}, there exists a $\frac{2}{3}$-separator $X$ such that $|X| < \binom{n-1}{2}$. In fact, by the original proof in \cite{GMII}, we can go further and assert that $X$ is a subset of a bag of $(B_x:x \in V(T))$.

Now $|G-X| = \binom{n}{2} - |X| > \binom{n-1}{1} = n-1$. By Lemma~\ref{lemma:part}, $V(G-X)$ has two parts $A$ and $B$ such that $\frac{1}{3}|G-X| \leq |A|,|B| \leq \frac{2}{3}|G-X|$ and there is no edge with an endpoint in $A$ and $B$. (Note that this bound on $|A|$ and $|B|$ is slightly weaker than in Lemma~\ref{lemma:part}, but has the benefit of being the same on both parts.) As $n \geq 6$, it follows that $|A|,|B| \geq 2$. By Theorem~\ref{theorem:ekr}, $V(G-X)$ is too large to be an independent set, and so it contains an edge, with both endpoints in $A$ or both endpoints in $B$.

Without loss of generality this edge is $\{1,2\}\{3,4\} \in A$. Then $B \subseteq \{\{1,3\},\{1,4\},\{2,3\},\{2,4\}\}$. If $B$ contains an edge, then $V(G-X) \subseteq \{\{1,2\},\{1,3\},\{1,4\},\{2,3\},\{2,4\},\{3,4\}\}$ and has maximum order $6$. Otherwise, without loss of generality, $B = \{\{1,3\},\{1,4\}\}$ and $A = \{\{3,4\},\{1,i\}| i \notin \{1,3,4\}\}$, so $|G-X| = n$. (Note $A$ must be exactly that set, or $|G-X|$ is too small.)

If $n \geq 7$, then $|G-X| \geq 7$ and the first case cannot occur. However in the second case, $|B| = 2 < \frac{1}{3}\cdot 7 \leq \frac{1}{3}n$. So neither case can occur, and we have forced a contradiction on either $|G-X|$ or $|B|$. This completes the proof when $n \geq 7$. Hence, let $n = 6$, and  note $|G-X|=6$ in either case.

Now we use the fact that $X$ is a subset of some bag $B_x$. Now for all $x \in V(T)$, $|B_x| \leq \binom{5}{2}-1 = 9$. As $|G-X|=6$, it follows $|X|=9$. Hence $X$ is exactly a bag of maximum order. For either choice of $G-X$, note that $A$ is a connected component. So there is some subtree of $T - x$ that contains all vertices of $A$. Let $y$ be the node of this subtree adjacent to $x$. Also note, for either choice of $G-X$, that each vertex of $X$ has a neighbour in $A$. So every vertex of $B_x$ is also in bag $B_y$, which contradicts our normalisation.


Thus, if $n \geq 6$, then $\tw(G) \geq \binom{n-1}{2}-1$. This completes the proof of Theorem~\ref{theorem:ktwo}.

\section{Open Questions}
We conjecture that Theorem~\ref{theorem:main} should also hold for smaller values of $n$.
\begin{conjecture}
\label{conjecture:full}
Let $G$ be a Kneser graph with $n \geq 3k$ and $k \geq 2$. Then $\tw(G) = \binom{n-1}{k}-1$.
\end{conjecture}
This conjecture follows directly from Theorem~\ref{theorem:ktwo} when $k=2$. The Petersen graph also shows that $n \geq 3k$ is a tight bound when $k=2$.

In general, we can determine a slightly better tree decomposition when $n < 3k-1$. Let $X = \{v \in V(G): 1 \in v\}$, and let $W$ be an independent set in $V(G)-X$ such that no two vertices of $W$ have a common neighbour in $X$. We define a tree decomposition for $G$ with underlying tree $T$ as follows. Let $r$ denote the root node of $T$, and let $r$ have one child node for each vertex in $W$ and each vertex in $X$ adjacent to no vertex in $W$. Label each of these child nodes by their associated vertex of $G$. Let each node labeled by a vertex $w \in W$ have one child node for each vertex of $N(w) \cap X$. Label each of those child nodes by their associated vertex of $G$, and note that since every vertex of $X$ has at most one neighbour in $W$, no vertex of $G$ labels more than one node of $T$.

Define the bag indexed by $r$ to be $V(G)-W-X$. Note this bag contains less than $\binom{n-1}{k}$ vertices when $W \neq \emptyset$. If a node is labeled by a vertex $v \in X$, let the corresponding bag be $N(v) \cup \{v\}$. These bags contain $\binom{n-k}{k}+1$ vertices. If a node is labeled by a vertex $w \in W$, let the corresponding bag be $\{w\} \cup \{u : uw \in E(G), 1 \notin u\} \cup \{u : ux \in E(G)$ where $xw \in E(G)$ and $1 \in x\}$. These bags contain less than $\binom{n-1}{k}$ vertices whenever $|W| \geq 2$, as they contain no vertex in $X$, and each contains only one vertex from $W$. This is a valid tree decomposition, but we omit the proof. When $|W| \geq 2$, the width of this tree decomposition is less than the width given by Lemma~\ref{lemma:gub}.


However, when $|W| \leq 1$, this tree decomposition has the same width as given by Lemma~\ref{lemma:gub}. We can construct $W$ such that $|W| \geq 2$ iff $n < 3k-1$. For example, let $W = \{\{2,\dots,(k+1)\},\{(k+1), \dots, 2k\}\}.$ If $n \leq 3k-2$, then any vertex of $X$ must be non-adjacent to at least one vertex of $W$. Alternatively, if $n \geq 3k-1$ and $|W| \geq 2$, then there exists two vertices $x,y \in W$ such that $|x \cup y| \leq 2k-1$. Then $X$ contains a vertex adjacent to both $x$ and $y$.
Hence, for general $n$, we cannot improve the lower bound on $n$ in Theorem~\ref{theorem:main} to $3k-2$ or below. This does leave a question about what may occur for $n=3k-1$. It is possible that Theorem~\ref{theorem:main} holds for $n \geq 3k-1$, with the Petersen graph as a single exception.

\subsection*{Acknowledgements}
Thanks to Alex Scott for helpful conversations, and for pointing out references \cite{almostisect,sperner,scott}.

\bibliographystyle{plain}
\bibliography{Kneserbib}

\end{document}